\documentclass[review]{elsarticle}

\usepackage{lineno,hyperref}
\modulolinenumbers[5]

\journal{Journal of \LaTeX\ Templates}
\usepackage{geometry}
\usepackage{tikz}
\usepackage{subfig}
\usepackage{graphicx}
\usepackage{comment,xcolor,colortbl,graphicx,mathrsfs,amsmath}
\usepackage{amsmath}
\usepackage{amssymb}
\usepackage{amsthm}
\usepackage{float}
\usepackage{enumerate}
\usepackage{arydshln}
\usepackage{bm}
\usepackage{url}
\newtheorem{theorem}{Theorem}
\newtheorem{lemma}{Lemma}
\newtheorem{remark}{Remark}

\newtheorem{definition}{Definition}
\newtheorem{proposition}{Proposition}
\newtheorem{corollary}{Corollary}
\newtheorem{example}{Example}
\begin{document}
\begin{frontmatter}
\title{Realization of the transient dynamics of dimension-varying control systems}
\tnotetext[cor1]{This work was supported by the Natural Science Basic Research Program of Shaanxi Province under Grant no. 2021JZ-12.}
\author[1]{Naqi Fan}

\author[1,2]{Lijun Zhang\corref{cor2}}
\ead{zhanglj7385@nwpu.edu.cn}
\cortext[cor2]{Corresponding author at: School of Marine Technology, Northwestern Polytechnical University, Xi'an {\rm710072}, China.}
\address[1]{College of Intelligent Systems Science and Engineering, Harbin Engineering University, Harbin, 150001, China.}
\address[2]{School of Marine Technology, Northwestern Polytechnical University, Xi'an, 710129, China.}
\begin{abstract}
Dynamic evolution behaviors of dimension-varying control systems often appear in the genetic regulatory network and the vehicle clutch system etc. An interesting and significant study on dimension-varying control systems is how to realize the dimension-varying (the transient dynamics) process smoothly between the different dimensional subsystems. The quotient space approach is considered as an effective tool to model the transient dynamics of dimension-varying control systems.
This paper investigates the realization problem of the transient dynamics of dimension-varying control systems. By revealing the structure of the controllable subspace for the linear system on quotient space, we propose the condition for the realization of the transient dynamics of dimension-varying control systems, based on which a new scheme for modeling the transient dynamics is given. Moreover, the proposed scheme justifies the existing result for the strategy for modelling the transient dynamics. A numeric example is given to illustrate our theoretical results.
\end{abstract}

\end{frontmatter}
\section{Introduction}
Dimension-varying control systems, also known as cross-dimensional control systems, are used for modeling complex systems with state spaces of different dimensions.
In practice, many dynamic behaviours can be modeled by a dimension-varying control system.
For example, in the internet or some other service-based networks, some users may join or withdraw every now and then \cite{2005Evolution}.
In a genetic regulatory network, cells may die or birth at any time \cite{2010Interactome,1981Application}.
Docking, undocking, departure and joining of spacecrafts  have ``short periods" of dimension-varying process \cite{Yang2013Spacecraft}.
During the dimension-varying process, the dimension-varying control system evolves from one model to another model of different dimensions.

In the past, researchers did some work on dimension-varying control systems, where they usually treated dimension-varying control systems as switching systems and investigated the stability and control design problem based on the switching system theory\cite{Hao2014Stabilization,Pakniyat2017On,Caines2017Hybrid,Yang2013Spacecraft}. However,
this scheme ignores the transient dynamics of the system during the dimension-varying process. In fact, this transient period may be long enough so that the dynamics during this process can not be ignored \cite{Yang2013Spacecraft,Jiao2014Modeling}.
For example, in automotive engineering, the vehicle clutch system shown as Figure \ref{clutch} can be described as 2-dimensional system or 1-dimensional system depending on whether the clutch is ``disengaged" or ``engaged" \cite{Temporelli2018Accurate,Cheng2020Equivalence}.
It takes about 1 second to complete the transient process of ``disengaged" or ``engaged".
The processes of docking, undocking, departure and joining in spacecrafts formation take even longer \cite{Pakniyat2017On,Caines2017Hybrid,Yang2013Spacecraft}. Therefore, it is significant to study the transient dynamics of dimension-varying control systems. Modeling and control design for the transient dynamics of dimension-varying control systems can improve the performance of mechanical or other systems.

However, for decades, mathematically few researches on the transient dynamics of dimension-varying control systems appear due to the lack of proper theoretical tool. Fortunately, Prof. Cheng \cite{Cheng2019From} in 2019 developed the theory of dimension-varying control systems, where he investigated the underlying mathematics of dimension-varying control systems and deliver a masterful exposition of dimension-varying control systems to model and analyse the dynamics of dimension-varying process of linear systems. Typically, based on quotient space theory, Cheng in \cite{Cheng2020Equivalence} proposed a projecting system of dimension-varying control systems on a developed quotient space which consists of the equivalent class of the state of different dimensional systems. In consequence, the study on the transient dynamics of dimension-varying control systems is treated as the investigation on the corresponding projecting systems on quotient space.

\begin{figure}[H]
  \centering
  \includegraphics[width=6cm]{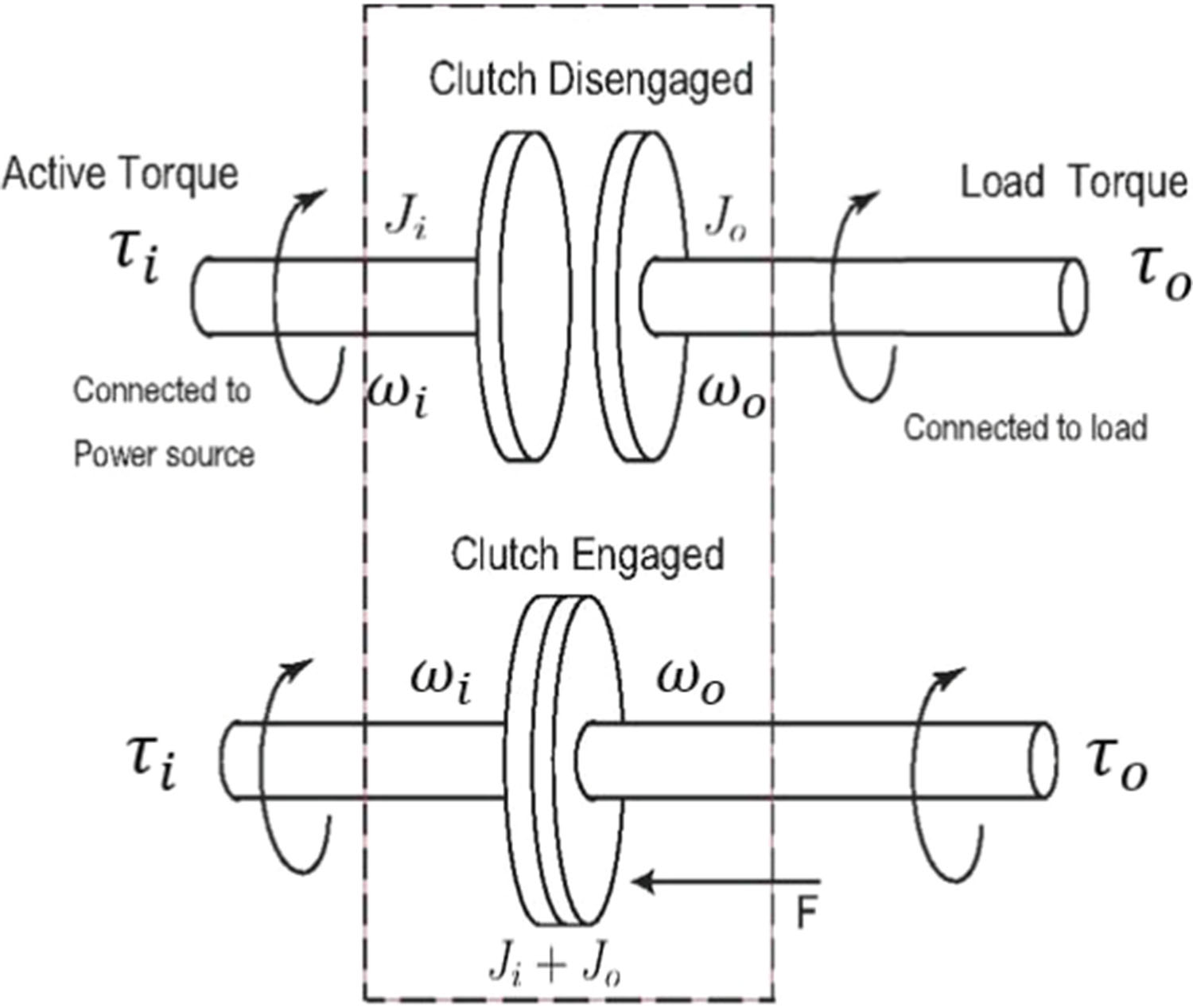}\\
  \caption{\small The vehicle clutch system. }\label{clutch}
\end{figure}

Specifically, for the given dimension-varying control system consisting of two linear subsystems with various dimensions:
\begin{equation}\label{system}
\begin{aligned}
&\Sigma_{1}: \dot{x}(t)=A_{1} x(t)+B_{1} u_{1}(t), \quad x \in \mathbb{R}^{p},\\
&\Sigma_{2}: \dot{y}(t)=A_{2} y(t)+B_{2} u_{2}(t), \quad y \in \mathbb{R}^{q},
 \end{aligned}
\end{equation}
where $x(t)$ and $y(t)$ are states of subsystems $\Sigma_{1}$ and $\Sigma_{2}$ respectively, $A_{1}\in \mathbb{R}^{p\times p}$, $A_{2}\in \mathbb{R}^{q\times q}$, $B_{1}\in \mathbb{R}^{p\times r}$ and $B_{2}\in \mathbb{R}^{q\times s}$ are  constant matrices that characterise the subsystems, $u_{i}(t)$ $(i=1, 2)$ stands for the controlled input of subsystem $\Sigma_{i}$. During the dimension-varying process, system \eqref{system} will evolve from $\Sigma_{1}$ to  $\Sigma_{2}$ (or $\Sigma_{2}$ to  $\Sigma_{1}$).
For system \eqref{system}, a well-known challenge is that different dimensions of $\Sigma_1$ and $\Sigma_2$ lead to an intractable modeling of the transient dynamics. In order to overcome this challenge, Cheng \cite{Cheng2018Linear, Cheng2019From} first proposed that the equivalence of linear systems on spaces of different dimensions and
constructed a quotient space based on the equivalence, and then modeled the transient dynamics of dimension-varying control systems by projecting the dimension-varying control system onto quotient space. As a result, the transient dynamics of system \eqref{system} can be realized  if one can design
$u_1(t)$ and $u_2(t)$ such that $\Sigma_{1}$ ($\Sigma_{2}$) can evolve to  $\Sigma_{2}$ ($\Sigma_{1}$) during  the transient period. Thereby, in theory, the transient process of dimension-varying control systems can be realized as smoothly as possible via the control design.

Recently the theory of dimension-varying control systems proposed by Cheng \cite{Cheng2019From} has attracted considerable attention. Zhang et al. \cite{Zhang2019Long} given a set of basis for the cross-dimensional state space of cross-dimensional systems to show it is of countably infinite dimension.
Feng et al. \cite{Feng2021On} considered the variation of a class of cross-dimensional linear systems and established an algorithm for calculating the state dimensions after a transition time.
Other new published research on dimension-varying control systems, please refer to \cite{Zhao2021On,Zhang2021Solution,Liu2021Relations,Zhang2021On}. However, to the best of our knowledge, other than Cheng's works \cite{Cheng2018Linear, Cheng2019From,Cheng2020Equivalence}, the above mentioned works have never referred to the realization problem
of the transient dynamics of dimension-varying control systems. What is worth saying, Cheng \cite{Cheng2020Equivalence} applied his theory to modeling the transient dynamics of vehicle clutch systems (Figure \ref{clutch}), based on which he had designed the control for dimension-varying process of clutch systems.
The control implementation for the transient dynamics of vehicle clutch systems in \cite{Cheng2020Equivalence} urges us a query: ~What the condition for ``the control implementation" is.

\emph{Cheng's works\cite{Cheng2018Linear, Cheng2019From,Cheng2020Equivalence} provide us a promising direction to investigate the transient dynamics of the dimension-varying control system, in which a fundamental theoretical framework has been built. However, as for the realization of the transient dynamics of dimension-varying control systems, several intriguing and challenging issues arise:
\begin{enumerate}
  \item Whether the realization problem of system \eqref{system} is equivalent to that of its projecting system?
  \item What conditions need to be satisfied for system \eqref{system} to realize transient process?
  \item What conditions should be satisfied to model the transient dynamics of dimension-varying control systems?
\end{enumerate}}

Motivated by the above issues, we proceed to study the realization problem of transient dynamics of dimension-varying control systems in our paper.
Specifically, from the viewpoint of the controllability theory of the system, this paper aims at providing a theoretical basis for modeling the transient dynamics of dimension-varying control systems on quotient space.
By analyzing the controllable subspace of the linear system on quotient space, we first show  that the realization of dimension-varying control systems can be equivalent to that of its projecting system on quotient space.
Specifically, we present the coordinate transformation of linear system on quotient space for controllability decomposition, in which we show that the relation between the controllability of dimension-varying control systems and that of its projecting systems.
Furthermore, we present a condition for the realization of the transient dynamics of dimension-varying control systems based on the controllability of the projecting system.
In particular, based on the proposed condition, a new scheme for modeling the transient dynamics of dimension-varying control systems is developed, based on which the realization of the transient dynamics of vehicle clutch systems in \cite{Cheng2020Equivalence} can be justified.

The structure of this paper is organized as follows:
In section \ref{S2}, we give some preliminaries on quotient space and formulate problems.
Section \ref{S3} analyzes the controllability of linear systems on quotient space and gives the coordinate transformation for controllability decomposition. The condition for the realization of the transient dynamics of dimension-varying control systems is presented in Section \ref{S4}. Section \ref{S5} presents the condition for modeling the transient dynamics of dimension-varying control systems, where we also recall an example in \cite{Cheng2019From} to illustrate the effectiveness of our proposed the theoretical results. The conclusion is drawn in Section \ref{S6}.

\section{Preliminaries and Problem Description}\label{S2}
In this section, we sketch fundamental definitions and concepts of quotient space, and introduce linear systems on quotient space.
For more details, the reader is suggested to refer to \cite{Cheng2019On}. At the end of this section, the considered problems about the realization of the transient dynamics of dimension-varying control systems are put forth.

Some notations are given first.
Symbol $\otimes$ is Kronecker product of matrices.
One-entry vector is $\mathbf{1}_{n}=[\underbrace{1, \cdots, 1}_{n}]^{T}$.
Denote the $k\times k$  matrix  with its entries being $\frac{1}{k}$ by $\mathbf{J}_{k}$.
The set of mix-dimensional vectors is defined as $\mathcal{V}:=\bigcup_{n=1}^{\infty} \mathcal{V}_{n}$, where  $\mathcal{V}_{n}$ is an $n$-dimensional vector space. For brevity, we assume $\mathcal{V}_{n}=\mathbb{R}^{n}$.
The set of all matrices is defined as $\mathcal{M}:=\bigcup_{m=1}^{\infty} \bigcup_{n=1}^{\infty} \mathcal{M}_{m\times n}$, where $\mathcal{M}_{m\times n}$ represents the set of $m \times n$
dimensional real matrices.

\subsection{Preliminaries}
\begin{definition}\cite{Cheng2019On}\label{vector-equiv}
1) Let $x, y\in \mathcal{V}$. $x$ and $y$ are said to be equivalent, denoted by $x\leftrightarrow y$, if there exist $\mathbf{1}_{\alpha}$ and $\mathbf{1}_{\beta}$, such
that
$x \otimes \mathbf{1}_{\alpha}=y \otimes \mathbf{1}_{\beta}.$

\noindent 2) The equivalent class is denoted by
$$\bar{x}=\{y \in \mathcal{V} \mid y \leftrightarrow x\}.$$

\noindent 3) A vector $x_{1} \in \bar{x}$ is irreducible, if there are no $y$ and $\mathbf{1}_{s}$, $s > 1$, such that $x_{1} = y \otimes \mathbf{1}_{s}$.
Then, all the elements in $\bar{x}$ can be expressed as $x_{i}=x_{1}\otimes\mathbf{1}_{i}, i=1, 2, 3, \ldots$, and  $x_{1}$ is an unique irreducible element.

\noindent 4) The quotient vector space of $\mathcal{V}$ under equivalence relation $\leftrightarrow$, denoted by $\Omega$, is
$$
\Omega:=\{\bar{x} \mid x \in \mathcal{V}\} .
$$
\end{definition}

\begin{definition}\cite{Cheng2019On}\label{def1}
Let $x\in \mathcal{V}_m$, $y\in \mathcal{V}_n$ and $t = m\vee n$ be the least common
multiple of $m$ and $n$.  An addition of
$x$ and $y$ is defined as $x \vec{|+} y:=\left(x \otimes \mathbf{1}_{t / m}\right)+\left(y \otimes \mathbf{1}_{t / n}\right) \in \mathcal{V}_t$.
Correspondingly, the subtraction is defined as $x \vec{\vdash} y:=x \vec{|+} (-y).$
Then, for $\bar{x},\bar{y}\in \Omega$ and $a\in \mathbb{R}$,
$$
\begin{aligned}
\bar{x} \vec{|+} \bar{y} &:=\overline{x \vec{|+} y}, \\
\bar{x} \vec{\vdash} &:=\overline{x \vec{\vdash} y}, \\
a \bar{x} &:=\overline{a x}.
\end{aligned}
$$
\end{definition}

Definition \ref{def1} shows the quotient space $\Omega$ is a vector space.

\begin{definition}\cite{Cheng2019On}\label{D2}
Let $A, B\in \mathcal{M}$. $A$ and $B$ are said to be equivalent, denoted by $A\thickapprox B$, if there exist $\mathbf{J}_{\alpha}$ and $\mathbf{J}_{\beta}$, such
that
$
A \otimes \mathbf{J}_{\alpha}=B \otimes \mathbf{J}_{\beta}.
$
The equivalent class is denoted by
$
\hat{A}=\{B \mid B \approx A\}.
$
Then the quotient space is denoted by
$
\Xi:=\mathcal{M} / \approx.
$
\end{definition}

\begin{definition}\cite{Cheng2019On}\label{secondSTP}
1)Let $A \in \mathcal{M}_{m \times n}, B \in \mathcal{M}_{p \times q}$, and $t = n\vee p$. Then the second semi-tensor product of $A$ and $B$, denoted
by $A \circ B$, is defined as
$$
A \circ B:=\left(A \otimes J_{t / n}\right)\left(B \otimes J_{t / p}\right).
$$
\noindent 2)Define a product on $\Xi$ as
$
\hat{A} \circ \hat{B}:=\widehat{A \circ B}.
$
\end{definition}

The following definition presents a vector equivalence for two matrices, which is used to describe linear systems on quotient space.
\begin{definition}\cite{Cheng2019On}\label{defi-mvec-equi}
Let $B, D\in \mathcal{M}$. $B$ and $D$ are said to be vector equivalent, denoted by $B\leftrightarrow D$, if there exist $\mathbf{1}_{\alpha}$ and $\mathbf{1}_{\beta}$, such that
$
B \otimes \mathbf{1}_{\alpha}=D \otimes \mathbf{1}_{\beta}.
$
The equivalent class of $B$ is denoted by
$
\bar{B}=\{D \mid D \leftrightarrow B\}.
$
\end{definition}
\begin{definition}\cite{Cheng2019On}
1) Let $A \in \mathcal{M}_{m \times n} \subset \mathcal{M}$, $x \in \mathcal{V}_{r} \subset \mathcal{V}$ and $t = n\vee r$. Then the product of $A$ and $x$, denoted by $\vec{\circ}$, is defined as
\begin{equation}
A \vec{\circ} x:=\left(A \otimes \mathbf{J}_{t / n}\right)\left(x \otimes \mathbf{1}_{t / r}\right).
\end{equation}
\noindent 2) The action of $\Xi$ on $\Omega$ is denoted as $\hat{A} \vec{\circ} \bar{x}:=\overline{A \vec{\circ} x}$.
\end{definition}

Viewing a matrix $B$ as a set of column vectors, the above definition can be easily extended to the case of two matrices $A$ and $B$:
\begin{definition}\label{ABcolum}
1) Let $A \in \mathcal{M}_{m \times n}$, $B \in \mathcal{M}_{r \times s}$ and $t = n\vee r$. Then
\begin{equation}
A \vec{\circ} B:=\left(A \otimes \mathbf{J}_{t / n}\right)\left(B \otimes \mathbf{1}_{t / r}\right).
\end{equation}
\noindent 2) The action on $\Xi$ is denoted as $\hat{A} \vec{\circ} \bar{B}:=\overline{A \vec{\circ} B}$.
\end{definition}

By Definition 2,  Cheng \cite{Cheng2020Equivalence} has defined the projecting system of the linear system on quotient space as follows.
\begin{definition}\cite{Cheng2020Equivalence}\label{D1}
\noindent 1) Consider a linear system
\begin{equation}\label{lifting}
\begin{aligned}
&\dot{x}=A x(t)+B u(t), \quad x(t) \in \mathbb{R}^{r}.
\end{aligned}
\end{equation}
The following system on quotient space $\Omega$ is
called the projecting system of (\ref{lifting})
\begin{equation}\label{projecting}
\begin{aligned}
&\dot{\bar{x}}(t)=\hat{A} \vec{\circ} \bar{x}(t)+\bar{B} u(t), \quad \bar{x}(t) \in \Omega.
\end{aligned}
\end{equation}
In turn, system \eqref{lifting} is called
the lifting system of \eqref{projecting} on $\mathbb{R}^{r}$, if $A \in \hat{A}$ and $B \in \bar{B}$.

\noindent 2)
Let $\Theta_{0}$ be a linear control system on quotient
space and $\Theta_{n}$ be its lifting system on $\mathbb{R}^{n}$. Then all the lifting systems
are said to be equivalent.

\end{definition}
\subsection{Problem Description}
Note that a system on quotient space is a set of equivalent systems of different dimensions, hence the systems of different dimensions can be ``lifted" up to a space of the identical dimension.
This makes it possible to model the transient dynamics of dimension-varying control systems on quotient space.
Therefore, by Definition \ref{D1}, the projection of system \eqref{system} on quotient space can be described as follows.
\begin{equation}\label{qsystem}
\begin{aligned}
\bar{\Sigma}_{1}: \dot{\bar{x}}(t)=\hat{A}_{1}\vec{\circ} \bar{x}(t)+\bar{B}_{1} u_{1}(t), \quad \bar{x} \in \bar{\mathbb{R}}^p,\\
\bar{\Sigma}_{2}: \dot{\bar{y}}(t)=\hat{A}_{2} \vec{\circ} \bar{y}(t)+\bar{B}_{2} u_{2}(t), \quad \bar{y} \in \bar{\mathbb{R}}^q,
 \end{aligned}
\end{equation}
where $\bar{\mathbb{R}}^p=\{\bar x|x\in {\mathbb{R}}^p\}$ and $\bar{\mathbb{R}}^q=\{\bar y|y\in {\mathbb{R}}^q\}$. For simplify, hereafter we use  the notation $\Omega$ to represent $\bar{\mathbb{R}}^p$ or $\bar{\mathbb{R}}^q$.
Based on system \eqref{qsystem}, we can convert the transient dynamics from $\Sigma_{1}$ to $\Sigma_{2}$ of system \eqref{system} into that of $\bar{\Sigma}_{1}$ to  $\bar{\Sigma}_{2}$ on quotient space during the transient period  $[t_0, t_e]$.
\begin{definition}\cite{Cheng2020Equivalence}\label{defin-realization-cite}
For system \eqref{system}, the transient dynamics is said to be properly realized from a given starting state ${x}(t_{0})\in \mathbb{R}^{p} $ to ${y}(t_{e})\in \mathbb{R}^{q}$ during the transient period $[t_0, t_e]$, if, for system \eqref{qsystem}, there are controls $u_{1}(t)$ and $u_{2}(t)$ such that a given starting state $\bar{x}(t_{0}) $ can be controlled to the state  $\bar{y}(t_{e}) $ on the quotient space $\Omega$.
\end{definition}
Here several intriguing and challenging problems arise:
\begin{enumerate}
  \item Whether the realization problem of system \eqref{system} is equivalent to that of system \eqref{qsystem}?
  \item What conditions need to be satisfied for system \eqref{system} to realize transient process?
  \item What conditions should be satisfied to model the transient dynamics of dimension-varying control systems?
\end{enumerate}
Now, we will discuss each of these problems in details.

\section{Some Analyses for Linear Systems on Quotient Space}\label{S3}
In this section, our purpose is to illustrate the realization problem of system \eqref{system} is equivalent to that of system \eqref{qsystem} by studying the controllable subspace of the projecting system \eqref{projecting} on quotient space.
Furthermore, we present the coordinate transformation of the projecting system which will support our sequent research.

\subsection{The controllable subspace of linear system on $\Omega$}

\begin{lemma}\label{L1}
Let $\mathcal{C}=\operatorname{span}\left\{B, A B, \cdots, A^{r-1} B\right\}
$ be the controllable subspace of system \eqref{lifting}, the controllable subspace of the projecting system \eqref{projecting},
denoted by $\bar{\mathcal{C}}$, is
$$\overline{\mathcal{C}}=\operatorname{span}\{
  \bar{B}, \overline{AB}, \overline{A^{2} B}, \cdots, \overline{A^{r-1} B}  \}$$
\end{lemma}
\begin{proof}
It is easy to calculate that the trajectory of the projecting system \eqref{projecting} with the initial $\bar{x}(t_0)$ is
\begin{equation}\label{a0}
\bar{x}(t)=
\mathrm{e}^{\hat{A}t} \vec{\circ} \bar{x}(t_0)+\int_{t_0}^{t} \mathrm{e}^{\hat{A} (t-\tau)} \vec{\circ}\bar{B} u(\tau) \mathrm{d} \tau.
\end{equation}
Let $\mathbf{0}:=\{[\underbrace{0,0, \cdots, 0}_n]^T \mid n=1,2, \cdots\}$.
Suppose $\bar{x}(t_0)$ is driven to the origin $\bar{x}(t)=\mathbf{0}$, then (\ref{a0}) is further written as
\begin{equation}\label{6}
\bar{x}(t_0)=
-\int_{t_0}^{t} \mathrm{e}^{\hat{A} (t_{0}-\tau)} \vec{\circ}\bar{B} u(\tau) \mathrm{d} \tau.
\end{equation}
It can be obtained from Definition \ref{D2} that the rank of $\hat{A}$ is equal to that of $A$. Using Cayley¨CHamilton theorem \cite{Rugh1996Linear}, we have the following Taylor expansion,
$$\mathrm{e}^{\hat{A}(t_{0}-\tau)}=\sum_{j=0}^{r-1} \lambda_{j}(t_{0}-\tau)\hat{A}^{j},$$
where $\hat{A}^j=\hat{A}^0\circ\hat{A}^1\circ\cdots\hat{A}^j$.

Let $\gamma_{j}=\int_{t_{0}}^{t}  \lambda_{j}(t_{0}-\tau)u(\tau) \mathrm{d} \tau$, (\ref{6}) is rewritten as
\begin{equation}\label{eq1}
\begin{aligned}
\bar{x}(t_0)
&=-\sum_{j=0}^{r-1} \hat{A}^{j}\vec{\circ} \bar{B} \gamma_{j}\\
&=-[\begin{array}{llll}
\bar{B} & \hat{A}\vec{\circ} \bar{B} & \cdots & \hat{A}^{r-1}\vec{\circ} \bar{B}
\end{array}]
[\begin{array}{cccc}
\gamma_{0} &\gamma_{1} &\cdots &\gamma_{r-1}
\end{array}]^T
\end{aligned}
\end{equation}
It follows from Definition \ref{secondSTP} that $\hat{A}^j=\widehat{A^{j}}$.
Together with Definition \ref{ABcolum}, we have $\hat{A}^{j}\vec{\circ} \bar{B}=\widehat{A^{j}}\vec{\circ} \bar{B}=\overline{A^{j}\vec{\circ} B}$.
Furthermore, note that the column number of $A$ is equal to the row number of $B$, we have
$\overline{A^{j}\vec{\circ}B}= \overline{A^{j}B}$.
Thus, \eqref{eq1} can be expressed as
$$\begin{aligned}
\bar{x}(t_0)&=-[\begin{array}{llll}
\bar{B} & \overline{A\vec{\circ}B} & \cdots &\overline{{A}^{r-1}\vec{\circ}B}
\end{array}]
[\begin{array}{cccc}
\gamma_{0} &\gamma_{1} &\cdots &\gamma_{r-1}
\end{array}]^T\\
&=-[\begin{array}{llll}
\bar{B} & \overline{AB} & \cdots &\overline{{A}^{r-1}B}
\end{array}]
[\begin{array}{cccc}
\gamma_{0} &\gamma_{1} &\cdots &\gamma_{r-1}
\end{array}]^T
\end{aligned}
$$

Note that $\Omega$ is a vector space (recall Definition \ref{def1}), the controllable subspace of system \eqref{projecting} can be expressed as
$$\overline{\mathcal{C}}=\operatorname{span}\{
  \bar{B}, \overline{AB}, \overline{A^{2} B}, \cdots, \overline{A^{r-1} B}  \}$$
\end{proof}

 In contrast with $\mathcal{C}$, we can see that elements of $\bar{\mathcal{C}}$ are equivalence classes of elements of $\mathcal{C}$, which implies that controllable states of the projecting system \eqref{projecting} are equivalent classes of controllable states of system \eqref{lifting}.

Lemma \ref{L1} shows us that, say, if the initial state $x(t_0)$ is controlled to the state $x(t_e)$ for system \eqref{lifting}, correspondingly $\bar{x}(t_0)$ can be controlled to $\bar{x}(t_e)$ for system \eqref{projecting} on $\Omega$.
This means that the study on the transient dynamics of system \eqref{system} may be performed by its corresponding projecting \eqref{qsystem} on quotient space. That is, for dimension-varying control systems, instead of studying the realization of the transient dynamics of dimension-varying control systems, we may study the corresponding problem on quotient space. In view of  this observation, we restrict attention to the projecting system \eqref{qsystem}.

\subsection{The coordinate transformation of linear system on $\Omega$}

To proceed with our discussion, we only need to focus on the transient dynamic of the projecting system \eqref{qsystem}. This subsection concerns the coordinate transformation of system \eqref{qsystem} aiming at investigating the controllability decomposition on quotient space, which paves the way for the study of the realization problem of the transient dynamics on quotient space.

\begin{definition}\cite{Cheng2019From}
Let $A \in \mathcal{M}_{m \times n} \subset \mathcal{M}, x \in \mathcal{V}_{r} \subset \mathcal{V}.$
Assume $t = n\vee r$. Then the product of $A$ and $x$ is defined as
\begin{equation}
A \overrightarrow{\ltimes} x:=\left(A \otimes I_{t / n}\right)\left(x \otimes \mathbf{1}_{t / r}\right).
\end{equation}
\end{definition}

Now, suppose $y\in \bar{y}$ and $y_{\pi}\in \bar{y}_{\pi}$ are two irreducible vectors of same dimension, where $\bar{y}, \bar{y}_{\pi}\in \Omega $  are two vectors on quotient space $\Omega$, as seen in Definition  \ref{vector-equiv}. we have the following Lemma.

\begin{lemma}\label{L2}

Let $T_{\Lambda}$ be a coordinate transformation matrix such that $y_{\pi}=T_{\Lambda}y$.
Then $\langle T_{\Lambda}\rangle$ is said to be a ``pseudo-coordinate transformation matrix" such that $\bar{y}_{\pi}=\langle T_{\Lambda}\rangle\overrightarrow{\ltimes}\bar{y}$, where $\langle T_{\Lambda}\rangle=\{T_{\Lambda}, T_{\Lambda}\otimes I_{2},\ldots, T_{\Lambda}\otimes I_{n}\}$.
\end{lemma}
\begin{proof}
It is sufficient to show that, for any $\xi\in \bar{y}$, $\xi_{\pi}=T_{\Lambda}\overrightarrow{\ltimes}x\in \bar{y}_{\pi}$ holds.

Since $\xi\in \bar{y}$, then  $\xi=y\otimes 1_{\alpha}$, where $y\in\bar{y}$ is a irreducible vector.
By Lemma \ref{L1}, we have
\begin{equation}\label{Pseudo_coord_trans}
\begin{aligned}\xi_{\pi}=&T_{\Lambda}\overrightarrow{\ltimes}\xi=T_{\Lambda}\overrightarrow{\ltimes}(y\otimes \mathbf{1}_{\alpha})\\=&(T_{\Lambda}\otimes I_{\alpha})(y\otimes \mathbf{1}_{\alpha})\\=&T_{\Lambda}y\otimes\mathbf{1}_{\alpha}\\=&y_{\pi}\otimes \mathbf{1}_{\alpha}.
\end{aligned}
\end{equation}
Hence we have $\xi_{\pi}\leftrightarrow y_{\pi}$ which implies $\xi_{\pi}\in \bar{y}_{\pi}$. I.e., since the arbitrariness of $\xi$, from the derivation of \eqref{Pseudo_coord_trans}, we have   $\bar{y}_{\pi}=\langle T_{\Lambda}\rangle\overrightarrow{\ltimes}\bar{y}$, where $\langle T_{\Lambda}\rangle$ is determined according to the second line of \eqref{Pseudo_coord_trans}.

In addition, the coordinate transformation matrix $T_{\Lambda}$ implies any element of $\langle T_{\Lambda}\rangle$ is non-singular so that $\langle T_{\Lambda}\rangle$ can be viewed as performing the coordinate transformation ability for the vector $\bar{y}$.
\end{proof}

We note that it is clear that $\langle T_{\Lambda}\rangle$ is not a  real coordinate transformation matrix, so we call it ``pseudo-coordinate transformation matrix".
For statement ease, hereafter the terminology  ``coordinate transformation matrix" is applied to $\langle T_{\Lambda}\rangle$ without cause confusion.

Then, using Lemma \ref{L2}, the coordinate transformation of projecting systems on $\Omega$ can be easily obtained.

\begin{lemma}\label{L3}
 Suppose that the lifting system (\ref{lifting}) is transformed into the following system under the coordinate transformation $x_{\pi}=Tx$,
\begin{equation}\label{pi}
\begin{aligned}
&\dot{x}_{\pi}=A_{\pi}(t) x_{\pi}(t)+B_{\pi}(t) u(t), \quad x_{\pi}(t) \in \mathbb{R}^{r},
\end{aligned}
\end{equation}
Then, correspondingly, its projecting system (\ref{projecting}) on quotient space under the ``coordinate transformation" $\bar{x}_{\pi}=\langle T\rangle\overrightarrow{\ltimes}\bar{x}$ can be converted into
\begin{equation}\label{barpi}
\begin{aligned}
&\dot{\bar{x}}_{\pi}(t)=\hat{A}_{\pi} \vec{\circ} \bar{x}_{\pi}(t)+\bar{B}_{\pi} u(t), \quad \bar{x}(t) \in \Omega.
\end{aligned}
\end{equation}
\end{lemma}

Lemma \ref{L2} and Lemma \ref{L3} provide us a proper ``coordinate transformation" to bring systems into the controllability decomposition on quotient space, which is significant for us to proceed our study on the the realization of the transient dynamics of dimension-varying control systems on quotient space.
\section{Realization of the transient dynamics of dimension-varying control systems on $\Omega$}\label{S4}
In this section, a condition is presented for the realization of the transient dynamics of dimension-varying control systems.
To formulate our main result precisely, we shall need some more notations and terminology.

Let $\mathcal{W}$ be a vector space.
Denote $\bar{\mathcal{W}}:=\{\bar{\omega}\mid \omega\in \mathcal{W}\}$.
It is easy to see $\text{dim}(\bar{\mathcal{W}})=\text{dim}(\mathcal{W})$.
\begin{definition}
Let $\mathcal{V}_m$ and $\mathcal{V}_n$ be two vector space.
The sum of $\mathcal{V}_m$ and $\mathcal{V}_n$ is defined as $\mathcal{V}_m |\vec{+} \mathcal{V}_n:=\left\{x |\vec{+} y| x \in \mathcal{V}_m, y \in \mathcal{V}_n\right\}$.
Then the sum of $\bar{\mathcal{V}}_m$ and $\bar{\mathcal{V}}_n$ is defined as $\bar{\mathcal{V}}_m |\vec{+} \bar{\mathcal{V}}_n:=\overline{\mathcal{V}_m |\vec{+} \mathcal{V}_n}.$
\end{definition}
Given a vector $\bm{\varepsilon}=(\varepsilon_1, \varepsilon_2, \ldots, \varepsilon_m)^T\in \mathcal{V}_m$ and a vector space $\mathcal{V}_n$, where $m\leq n$. The \emph{embedded mapping} of $\bm{\varepsilon}$ on $\mathcal{V}_n$ is defined as $\varphi_{n}(\bm{\varepsilon}):=(\varepsilon_1, \varepsilon_2, \ldots, \varepsilon_m, \underbrace{0, 0, \ldots, 0}_{ n-m} )^T$, then we denote the embedded space by $\mathcal{V}_n^m =\text{span}\{\varphi_{n}(\bm{\varepsilon})\}$.
That is, the missing elements of $\bm{\varepsilon}$ is filled up with $0$.
\begin{definition}
Let $\mathcal{V}_m$ and $\mathcal{V}_n$ be two vector space and $m\leq n$.
 the direct sum  of $\mathcal{V}_n^m$ and $\mathcal{V}_n$ is defined as $\mathcal{V}_n^m \vec{\oplus} \mathcal{V}_n:=\left\{\varphi_{n}(x) + y| \varphi_{n}(x) \cap y=\emptyset, x \in \mathcal{V}_m, y \in \mathcal{V}_n\right\}$.
Then the direct sum of $\bar{\mathcal{V}}_n^m$ and $\bar{\mathcal{V}}_n$ is defined as $\bar{\mathcal{V}}_n^m \vec{\oplus} \bar{\mathcal{V}}_n:=\overline{\mathcal{V}_n^m \vec{\oplus} \mathcal{V}_n}$.
\end{definition}

Without loss of generality, we assume $p<q$ for system \eqref{system}.
Then a condition for realization of transient dynamics of dimension-varying control system \eqref{qsystem} is given as follows.

\begin{theorem}\label{T1}
Consider system \eqref{qsystem}.
Let $\bar{\mathcal{C}}_{i}, i=1, 2$ be controllable subspace of $\bar{\Sigma}_{i}$.
Then $\bar{x}(t_{0})$ can be controlled to a state $\bar{y}(t_{e})$ during the transient period $[t_0,t_e]$ on quotient space $\Omega$ if there exists
$\bar{\mathcal{C}}_{2^*}\subset\bar{\mathcal{C}}_{2}$ such that
$$\bar{\mathcal{C}}_{1}\vec{\oplus} \bar{\mathcal{C}}_{2^*}=\bar{\mathbb{R}}^{q},$$ where $\bar{\mathbb{R}}_{q}=\{\bar y~|~y\in \mathbb{R}^q\}$.
\end{theorem}
\begin{proof}

With a mild modification, the system $\Sigma_{1} $ might be described as follows:
 \begin{equation}\label{pie}
\left\{\begin{array}{l}
\dot{x}=A x+B u_1 , \quad x \in \mathbb{R}^{p},\\
\dot{x^{c}}=\mathbf{0}_{q-p}, \quad x^{c} \in \mathbb{R}^{q-p},
\end{array}\right.
\end{equation}
where $\mathbf{0}_{q-p}=[\underbrace{0,\ldots,0}_{q-p}]$.
Let $w=(x, x^{c})$, then the projecting system of \eqref{pie} on quotient space $\Omega$ is denoted by $$\bar{\Sigma}_{3}: \dot{\bar{w}}(t)=\hat{A}_{3}\vec{\circ} \bar{w}(t)+\bar{B}_{3} u_{1}(t), \quad \bar{w} \in \Omega.$$ Note that  $\bar{\Sigma}_{3}$ has the same controllable space as $\bar{\Sigma}_{1}$. Then the transient dynamics of system \eqref{qsystem} is identical to the transient dynamics of the the following system.

\begin{equation}\label{qsystem1}
\begin{aligned}
\bar{\Sigma}_{3}: \dot{\bar{w}}(t)=\hat{A}_{3}\vec{\circ} \bar{w}(t)+\bar{B}_{3} u_{1}(t), \quad \bar{w} \in \Omega,\\
\bar{\Sigma}_{2}: \dot{\bar{y}}(t)=\hat{A}_{2} \vec{\circ} \bar{y}(t)+\bar{B}_{2} u_{2}(t), \quad \bar{y} \in \Omega.
 \end{aligned}
\end{equation}
Therefore, we next study the transient dynamics of system \eqref{qsystem1}.

For system \eqref{qsystem1}, by Definition \ref{defin-realization-cite}, we need to show that there exist controls $u_{1}$ and $u_{2}$ such that the transient dynamics of system \eqref{qsystem1} can be realized from a given state $\bar{w}(t_0)$ of system $\bar{\Sigma}_{3}$ to the state $\bar{y}(t_e)$ of system $\bar{\Sigma}_{2}$ during the transient period $[t_0,t_e]$.

Let $\mathcal{C}_{3}$ and  $\bar{\mathcal{C}}_{3}$ be the controllable subspace of system \eqref{pie} and system $\bar{\Sigma}_{3}$, respectively. With the assumption $p<q$, it can be seen that $\operatorname{dim}(\mathcal{C}_{3})<q.$

Next we consider two cases.

Case 1: $\operatorname{dim}(\bar{\mathcal{C}}_{2})=q.$

Since $\operatorname{dim}(\mathcal{C}_{3})<q$, there exist a coordinate transformation $\mathbf{w}=T_1w$ such that system \eqref{pie} can be converted into
\begin{equation}
\left\{\begin{array}{l}
\dot{\mathbf{w}}_{1}=A_{3}^{11} \mathbf{w}_{1}+A_{3}^{12} \mathbf{w}_{2}+B_{3} u_{1}, \\
\dot{\mathbf{w}}_{2}=A_{3}^{22}\mathbf{w}_{2}.
\end{array}\right.
\end{equation}
Then, by Lemma \ref{L3}, accordingly there exist the coordinate transformation $\bar{\mathbf{w}}=\langle T_1\rangle\vec{\ltimes}\bar{w}$ such that $\bar{\Sigma}_3$ is converted into
\begin{equation}\label{xsys}
\left\{\begin{array}{l}
\dot{\bar{\mathbf{w}}}_{1}=\hat{A}_{3}^{11}\vec{\circ} \bar{\mathbf{w}}_{1}+\hat{A}_{3}^{12} \vec{\circ}\bar{\mathbf{w}}_{2}+\bar{B}_{3} u_{1}, \\
\dot{\bar{\mathbf{w}}}_{2}=\hat{A}_{3}^{22}\vec{\circ}\bar{\mathbf{w}}_{2}.
\end{array}\right.
\end{equation}
Note that $\operatorname{dim}(\bar{\mathcal{C}}_{2})=q$,  similarly, under the ``coordinate transformation" $\bar{\mathbf{y}}=\langle T_1\rangle\vec{\ltimes}\bar{y}$, $\bar{\Sigma}_{2}$ is converted onto
\begin{equation}\label{ysys}
\left\{\begin{array}{l}
\dot{\bar{\mathbf{y}}}_{1}=\hat{A}_{2}^{11} \vec{\circ}\bar{\mathbf{y}}_{1} +\hat{A}_{2}^{12}\vec{\circ}\bar{\mathbf{y}}_{2}+\bar{B}_{2}^{1}u_{2}\\
\dot{\bar{\mathbf{y}}}_{2}=\hat{A}_{2}^{21} \vec{\circ}\bar{\mathbf{y}}_{1}+\hat{A}_{2}^{22} \vec{\circ}\bar{\mathbf{y}}_{2}+\bar{B}_{2}^{2}u_{2}
\end{array}\right.
\end{equation}
Let $\bar{w}(t_0)=(\bar{\mathbf{w}}_{1}(t_{0}), \bar{\mathbf{w}}_{2}(t_0))$ and $\bar{y}(t_{e})=(\bar{\mathbf{y}}_{1}(t_{e}),\bar{\mathbf{y}}_{2}(t_{e}))$. Then it is enough to show that the transient dynamics can be realized from system \eqref{xsys} to system \eqref{ysys}. Clearly $\bar{\mathbf{w}}_{1}(t_{0})$ is controllable. Next, we only need to consider how to control $\bar{\mathbf{w}}_{2}(t_{0})$ to $\bar{\mathbf{y}}_{2}(t_{e})$. Note that $(\hat{A}_{2}^{22}, \bar{B}_{2}^{2})$ is controllable in \eqref{ysys}, by Definition \ref{defin-realization-cite},
there is $u_2$ such that $\bar{\mathbf{w}}_{2}(t_0)$ is controlled to $\bar{\mathbf{y}}_{2}(t_{e})$.
Thus, $\bar{w}(t_{0})$ can be controlled to $\bar{y}(t_{e})$.

Case 2: $\operatorname{dim}(\bar{\mathcal{C}}_{2})<q.$

Let $r=\operatorname{dim}(\bar{\mathcal{C}}_{1})$, clearly $\bar{\mathbf{w}}_{2}(t)\in \bar{\mathbb{R}}^{q-p}\vec{\oplus} \bar{\mathbb{R}}^{p-r}$. By the assumption
$\bar{\mathcal{C}}_{1}\vec{\oplus}\bar{\mathcal{C}}_{2^*}=\bar{\mathbb{R}}_{q}$ with $\bar{\mathcal{C}}_{2^*}\subset\bar{\mathcal{C}}_{2}$, it follows $\bar{\mathcal{C}}_{2^*}= \bar{\mathbb{R}}^{q-p}\vec{\oplus} \bar{\mathbb{R}}^{p-r}$. i.e., $\bar{\mathbf{w}}_{2}(t)\subset \bar{\mathcal{C}}_{2}$.  So we have
\begin{equation}\label{coordinate_decompostion}
\bar{\mathbb{R}}^{q}=\bar{\mathcal{C}}_{1}\vec{\oplus}\bar{\mathbb{R}}^{q-p}\vec{\oplus} \bar{\mathbb{R}}^{p-r}.
\end{equation}
Based on \eqref{coordinate_decompostion}, by Lemma \ref{L3}, there is a proper ``coordinate transformation" $\bar{\mathbf{w}}=\langle T_2\rangle\vec{\ltimes}w$ to
further split $\bar{\Sigma}_3$ into
\begin{equation}\label{12}
\left\{\begin{array}{l}
\dot{\bar{\mathbf{w}}}_{1}=\hat{A}_{3}^{11} \vec{\circ} \bar{{\mathbf{w}}}_{1}+\hat{A}_{3}^{12}\vec{\circ}\bar{\mathbf{w}}_{2}^{1}+\hat{A}_{3}^{13}\vec{\circ}\bar{\mathbf{w}}_{2}^2+\bar{B}_{3}^{1}u_{1} ,\quad{\bar{\mathbf{w}}}_{1}\in \bar{\mathbb{R}}^{r},\\
\dot{\bar{\mathbf{w}}}_{2}^{1}=\hat{A}_{3}^{22}\vec{\circ}\bar{\mathbf{w}}_{2}^{1}+\hat{A}_{3}^{23}\vec{\circ}\bar{\mathbf{w}}_{2}^2,
~~~~~~~~~~~~~~~~~~~~~~~~~~~~~~~~\quad{\bar{\mathbf{w}}}_{2}^{1}\in\bar{\mathbb{R}}^{p-r},\\
\dot{\bar{\mathbf{w}}}_{2}^2=\hat{A}_{3}^{33}\vec{\circ}\bar{\mathbf{w}}^2_{2},
~~~~~~~~~~~~~~~~~~~~~~~~~~~~~~~~~~~~~~~~~~~~~~~~~~~\quad{\bar{\mathbf{w}}}_{2}^2\in\bar{\mathbb{R}}^{q-p}.
\end{array}\right.
\end{equation}
Using the same ``coordinate transformation matrix" $\langle T_2\rangle$ , $\bar{\Sigma}_2$ can be split into
\begin{equation}\label{Tysys}
\begin{aligned}
\left\{\begin{array}{l}
\dot{\bar{\mathbf{y}}}_{1}^{1}=\hat{A}_{2}^{11}\vec{\circ} \bar{\mathbf{y}}_{1}^{1}, ~~~~~~~~~~~~~~~~~~~~~~~~~~~~~~~~~~~~~~~~~~~~~~~~\quad {\bar{\mathbf{y}}}_{1}^{1}\in \bar{\mathbb{R}}^{r},\\
\dot{\bar{\mathbf{y}}}_{1}^{2}=\hat{A}_{2}^{21}\vec{\circ} \bar{\mathbf{y}}_{1}^{1}+\hat{A}_{2}^{22}\vec{\circ}\bar{\mathbf{y}}_{1}^{2}+\hat{A}_{2}^{23}\vec{\circ}\bar{\mathbf{y}}_{2}+\bar{B}_{2}^{1}u_{2},\quad {\bar{\mathbf{y}}}_{1}^{2}\in\bar{\mathbb{R}}^{p-r},\\
\dot{\bar{\mathbf{y}}}_{2}=\hat{A}_{2}^{31}\vec{\circ} \bar{\mathbf{y}}_{1}^{1}+\hat{A}_{2}^{32}\vec{\circ}\bar{\mathbf{y}}_{1}^{2}+\hat{A}_{2}^{33}\vec{\circ}\bar{\mathbf{y}}_{2}+\bar{B}_{2}^{2}u_{2},\quad {\bar{\mathbf{y}}}_{2}\in\bar{\mathbb{R}}^{q-p}.
\end{array}\right.
\end{aligned}
\end{equation}

Now we analyze the transient dynamics of system \eqref{12} and system \eqref{Tysys}.

 For system \eqref{12}, since $(\hat{A}_{3}^{11}, \bar{B}_{3}^{1})$ is controllable, there is the control $u_{1}$ such that the state $\bar{\mathbf{w}}_{1}(t) $ is controlled to $\bar{\mathbf{y}}_{1}^{1}(t_e)$. Again since $(\hat{A}_{2}^{22}, \hat{B}_{2}^{1})$ and $(\hat{A}_{2}^{33}, \hat{B}_{2}^{2})$ are controllable, there is the control $u_{2}$ such that the states $\bar{\mathbf{w}}^{1}_{2}(t)$ and $\bar{\mathbf{w}}_{2}^{2}(t)$ are controllable. That is, for system \eqref{12} and system \eqref{Tysys}, we can find control $u_{1}$ and $u_{2}$ such that the transient dynamic from the starting state $\bar{w}(t_0)$ to the destination state $\bar{y}(t_e)$ during the transient period $[t_0,t_e]$.
\end{proof}
As a direct consequence, the version of Theorem \ref{T1} on Euclidean space can be stated as follows.
\begin{corollary}\label{C1}
Consider system \eqref{system}.
Let $\mathcal{C}_{i}, i=1, 2$ be the controllable subspace of $\Sigma_{i}$.
Then the dimension transience of system \eqref{system} is properly realized if there exists
$\mathcal{C}_{2^*}\subset\mathcal{C}_{2}$ such that
$$\mathcal{C}_{1}\vec{\oplus} \mathcal{C}_{2^*}=\mathbb{R}^{q}.$$
\end{corollary}

\section{Condition for Modeling the Transient Dynamics of Dimension-Varying Systems}\label{S5}
A necessary condition for modeling the transient dynamics is proposed in this section, based on which a new scheme for modeling the transient dynamics of dimension-varying control systems is developed.
Accordingly, a numerical example is given to illustrate the effectiveness of our proposed the theoretical results.

Based on Definition \ref{defin-realization-cite}, the transient dynamics of dimension-varying control system \eqref{qsystem} on quotient space $\Omega$ is modeled as the following unified form:
\begin{equation}\label{tsystem}
\dot{\bar{z}}=\bar{F}(\bar{z}, u_{1}, u_{2}), \bar{z}\in \Omega.
\end{equation}
\eqref{tsystem} is called \emph{transient dynamics model}. Our purpose is to realize the transient process of \eqref{qsystem} with the help of the transient dynamics model \eqref{tsystem} on $\Omega$.
Based on Theorem \ref{L1}, a necessary condition for modeling the transient dynamics can be obtained naturally.
\begin{corollary}\label{C2}
With the same notations as Theorem \ref{T1}. Let $\bar{\mathcal{C}_{z}}$ be the the controllable subspace of \eqref{tsystem}. Then a necessary condition for modeling the transient dynamics of system \eqref{qsystem} is as follows:
\begin{equation}\label{equa_trans_dyna}
  \bar{\mathcal{C}}_{1}+\bar{ \mathcal{C}}_{2}\subset \bar{\mathcal{C}}_{z}.
\end{equation}

\end{corollary}

Next, we will demonstrate that the modelling strategy of the transient dynamics of dimension-varying control systems  proposed by Prof. Cheng \cite{Cheng2019From} satisfies the condition of Corollary \ref{C2}.

Let $n=p\vee q$ be the least common multiple of $p$ and $q$. Then
\eqref{qsystem} can be ``lifted"  to $\mathbb{R}^{n}$ as
\begin{equation}\label{lsystem}
\begin{gathered}
\Sigma_z^{1}: \dot{z}(t)=\mathbf{A_{1}} z(t)+\mathbf{B_{1}} u_{1}(t),  \quad z(t)  \in \mathbb{R}^{n},\\
\Sigma_z^{2}: \dot{z}(t)=\mathbf{A_{2}} z(t)+\mathbf{B_{2}} u_{2}(t),  \quad z(t)  \in \mathbb{R}^{n},
\end{gathered}
\end{equation}
where
$$
\begin{gathered}
\mathbf{A_{1}}=A_{1}\otimes \mathbf{J}_{n/p},
\mathbf{B_{1}}= B_{1}\otimes \mathbf{1}_{n/p},\\
\mathbf{A_{2}}=A_{2}\otimes \mathbf{J}_{n/q}, \mathbf{B_{2}}=B_{2}\otimes \mathbf{1}_{n/q}.
\end{gathered}
$$
Cheng \cite{Cheng2019From} modelled the transient dynamics of system \eqref{lsystem} as a linear combination form  as follows:
\begin{equation}\label{csystem}
\dot{z}(t)=\left[\mu \mathbf{A}_{1}+(1-\mu) \mathbf{A}_{2}\right] z(t)+\mu \mathbf{B}_{1} u_1+(1-\mu) \mathbf{B}_{2} u_{2}.
\end{equation}
where $\mu=\frac{m_{1}}{m_{1}+m_{2}}$, $m_{1}$ and $m_{2}$ are ``formal masses" of the two systems.
For notational convenience, let
$\mathbf{A}^*=\mu \mathbf{A}_1+(1-\mu) \mathbf{A}_2 $, $\mathbf{B}_{1}^*=\mu \mathbf{B}_{1}$ and $\mathbf{B}_{2}^*=1-\mu \mathbf{B}_{2}$.
Then the transient dynamics model \eqref{tsystem} of system \eqref{lsystem} can be specified as
\begin{equation}\label{zsystem}
\dot{\bar{z}}(t)=\hat{\mathbf{A}}^* \vec{\circ} \bar{z}(t)+\bar{\mathbf{B}}_{1}^* u_1+\bar{\mathbf{B}}_{2}^* u_{2}.
\end{equation}

\begin{proposition}\label{z-theo}
With the same notations as Theorem \ref{T1}. Let $\bar{\mathcal{C}}^*_{z}$ be the controllable subspace of the  transient dynamics \eqref{zsystem}. Then for system \eqref{lsystem}
$$\bar{\mathcal{C}}_{1}+\bar{ \mathcal{C}}_{2}\subset \bar{\mathcal{C}}^*_{z}.$$
\end{proposition}
\begin{proof}
The solution of (\ref{zsystem}) with $\bar{z}(0)=\bar{z}_{0}$ is given by
\begin{equation*}
\bar{z}(t)=
\mathrm{e}^{\hat{\mathbf{A}}^*(t-t_{0})} \vec{\circ} \bar{z}_{0}+\int_{t_{0}}^{t} \mathrm{e}^{\hat{\mathbf{A}}^* (t-\tau)} \vec{\circ}\bar{\mathbf{B}}_{1}^* u_{1}(\tau) \mathrm{d} \tau+\int_{t_{0}}^{t} \mathrm{e}^{\hat{\mathbf{A}}^* (t-\tau)} \vec{\circ}\bar{\mathbf{B}}_{2}^* u_{2}(\tau) \mathrm{d} \tau
\end{equation*}
Similar to the proof of Lemma \ref{L1}, the controllable subspace of \eqref{zsystem} is
$$\begin{aligned}
\bar{\mathcal{C}}^*_{z}&=\operatorname{span}\{[\begin{array}{llll}
\bar{\mathbf{B}}^*_{1} & \overline{\mathbf{A}^*\mathbf{B}^*_{1}} & \cdots & \overline{(\mathbf{A}^*)^{n-1} \mathbf{B}^*_{1}}
\end{array}]\}+\operatorname{span}\{[\begin{array}{llll}
\bar{\mathbf{B}}^*_{2} & \overline{\mathbf{A}^*\mathbf{B}^*_{2}} & \cdots & \overline{(\mathbf{A}^*)^{n-1} \mathbf{B}^*_{2}}
\end{array}]\},
\end{aligned}$$
Briefly,
$$\bar{\mathcal{C}}^*_{z}=\operatorname{span}\{\bar{\mathbf{C}}_{1}^*\}+\operatorname{span}\{\bar{\mathbf{C}}_{2}^*\}.$$
We note that $\bar{\mathcal{C}}_{1}\subset \operatorname{span}\{\bar{\mathbf{C}_{1}^*}\}$ and
$\bar{\mathcal{C}}_{2}\subset \operatorname{span}\{\bar{\mathbf{C}}_{2}\}$.
Thus,
$$\bar{\mathcal{C}}_{1}+\bar{ \mathcal{C}}_{2}\subset \bar{\mathcal{C}}^*_{z}.$$
\end{proof}
\begin{remark}\label{R3}
We note that the transient dynamics model \eqref{csystem} proposed in \cite{Cheng2019From} is represented as a convex combination form. However, from the proof of Proposition \ref{z-theo}, we can relax the restriction of the parameter $\mu$ to a generalized case. In fact, it is enough that to design the transient dynamics model of system  \eqref{lsystem} satisfies condition \eqref{equa_trans_dyna}. For example, we can choose the two parameters $\alpha, \beta\in \mathbb{R^+}$ such that the transient dynamics
\begin{equation}\label{last}
\dot{z}(t)=\left[\alpha \mathbf{A}_{1}+\beta \mathbf{A}_{2}\right] z(t)+\alpha \mathbf{B}_{1} u_1+\beta \mathbf{B}_{2} u_{2},
\end{equation}
satisfies \eqref{equa_trans_dyna}.
\end{remark}

The following numerical example from Cheng \cite{Cheng2019From} is illustrated to verify the condition of Corollary \ref{C2}. In particular, the parameters of the designed transient dynamics model need not satisfy the convex combination constraint, as required in Cheng \cite{Cheng2019From}.
\begin{example}\label{E1}\cite{Cheng2019From}
 Consider a dimension-varying control system, which
has two models $(A_{1},B_{1})$ and $(A_{2},B_{2})$ as
\begin{equation}\label{example}
\begin{aligned}
& \dot{x}(t)=A_{1} x(t)+B_{1} u_{1}(t), \quad x \in \mathbb{R}^{2},\\
& \dot{y}(t)=A_{2} y(t)+B_{2} u_{2}(t), \quad y \in \mathbb{R}^{3},
 \end{aligned}
\end{equation}
where
\end{example}
$$
\begin{array}{ll}
A_{1}=\left[\begin{array}{cccccc}
0 & 1 \\
0 & 0
\end{array}\right],  & B_{1}=\left[\begin{array}{cccccc}
0  \\
1
\end{array}\right], \\
A_{2}=\left[\begin{array}{cccccc}
0 & 0 & 1 \\
0 & 0 & 0 \\
0 & 1 & 0
\end{array}\right], & B_{2}=\left[\begin{array}{cccccc}
0  \\
1 \\
0
\end{array}\right].
\end{array}
$$

Let $\mathcal{C}_{1}$ and $\mathcal{C}_{2}$ be controllable subspaces of $(A_{1}, B_{1})$ and $(A_{2}, B_{2})$, respectively.
Then, we have
$$\mathcal{C}_{1}=\operatorname{span}\left\{\left[\begin{array}{l}0 \\ 1 \\  \end{array}\right],\left[\begin{array}{l}1 \\ 0 \end{array}\right]\right\}, \mathcal{C}_{2}=\operatorname{span}\left\{\left[\begin{array}{l}0 \\ 1 \\ 0\end{array}\right],\left[\begin{array}{l}1 \\ 0 \\ 0 \end{array}\right],\left[\begin{array}{l}0 \\ 0 \\ 1\end{array}\right]\right\}.$$
According to \eqref{last}, let $\alpha=\frac{3}{2}, \beta=\frac{1}{2}$. The transient dynamics of (\ref{example}) is calculated as follows
\begin{equation*}\label{CZ1}
 \dot{z}(t)=\left[\begin{array}{cccccc}
0 & 0 & 0 & 1/2 & 3/4 & 3/4 \\
0 & 0 & 0 & 1/2 & 3/4 & 3/4 \\
0 & 0 & 0 & 1/2 & 1/2 & 1/2 \\
0 & 0 & 0 & 0 & 0 & 0 \\
0 & 0 & 1/4& 1/4 & 0 & 0 \\
0 & 0 & 1/4 & 1/4 & 0 & 0
\end{array}\right] z(t)+\left[\begin{array}{l}
0 \\
0 \\
0 \\
3/2 \\
3/2 \\
3/2
\end{array}\right] u_{1}(t)+\left[\begin{array}{l}
0 \\
0 \\
1/2 \\
1/2 \\
0 \\
0
\end{array}\right] u_{2}(t), \quad z \in \mathbb{R}^{6}.
\end{equation*}
Its controllability matrix is
$$C_{z_{1}}=\left[
    \begin{array}{cccccccccccc}
      0 & 3 & 9/16 & 27/32 & 9/64 & 27/128 & 0 & 1/4 & 3/8 & 3/32 & 3/32 & 3/128\\
       0 & 3 & 9/16 & 27/32 & 9/64 & 27/128 & 0 & 1/4 & 3/8 & 3/32 & 3/32 & 3/128\\
      0 & 4/9 & 3/8 &9/16 & 3/32 & 9/64 & 1/2 & 1/4 & 1/4 & 1/16& 1/16 & 1/64 \\
      3/2 & 0 & 0 & 0 & 0 & 0 & 1/2 & 0 & 0 & 0 & 0 & 0 \\
      3/2 &  3/8 &  9/16 &  3/32 &  9/64 &  27/128 & 0 &  1/4 &  3/8 &  3/32 &  3/32 &  3/128 \\
      3/2 &  3/8 &  9/16 &  3/32 &  9/64 &  27/128 & 0 &  1/4 &  3/8 &  3/32 &  3/32 &  3/128 \\
    \end{array}
  \right].$$
Then, we obtain the controllable subspace of (\ref{example}) is
$$\mathcal{C}_{z_{1}}=\operatorname{span}\left\{\left[\begin{array}{l}0 \\ 0 \\ 0 \\ 3/2 \\ 3/2 \\ 3/2\end{array}\right],\left[\begin{array}{l}3 \\ 3 \\ 9/4 \\ 0 \\ 3/8 \\ 3/8\end{array}\right],\left[\begin{array}{l}0 \\ 0 \\ 1/2 \\ 1/2 \\ 0 \\ 0\end{array}\right],\left[\begin{array}{l}1/4 \\ 1/4 \\ 1/4 \\ 0 \\ 1/4 \\ 1/4\end{array}\right]\right\}.$$
Thus, we can verify that
$$\begin{aligned}
\bar{\mathcal{C}}_{1}+\bar{\mathcal{C}}_{2}
=&\operatorname{span}\left\{\overline{\left[\begin{array}{l}0 \\1 \\ \end{array}\right]}, \overline{\left[\begin{array}{l}1\\ 0\end{array}\right]}\right\}+
\operatorname{span}\left\{\overline{\left[\begin{array}{l}0 \\ 1 \\ 0 \\ \end{array}\right]}, \overline{\left[\begin{array}{l}1\\ 0 \\ 0\end{array}\right]}, \overline{\left[\begin{array}{l}0 \\0 \\ 1\end{array}\right]}\right\}\\
=&\operatorname{span}\left\{\overline{\left[\begin{array}{l}0 \\0 \\ 0\\ 1 \\ 1 \\ 1 \\ \end{array}\right]}, \overline{\left[\begin{array}{l}1\\ 1 \\ 1 \\0 \\0 \\ 0\end{array}\right]}\right\}+
\operatorname{span}\left\{\overline{\left[\begin{array}{l}0 \\0 \\ 1\\ 1 \\ 0 \\ 0 \\ \end{array}\right]}, \overline{\left[\begin{array}{l}1\\ 1 \\ 0 \\0 \\0 \\ 0\end{array}\right]}, \overline{\left[\begin{array}{l}0 \\0 \\0 \\ 0 \\ 1\\ 1\end{array}\right]}\right\}
\subseteq\bar{\mathcal{C}}_{z_{1}}.
\end{aligned}$$

In addition, it is easy to see that there exists $\mathcal{C}_{2^*}=\operatorname{span}\left\{\left[\begin{array}{lll}0 & 0  & 1\end{array}\right]^{T}\right\}\subseteq\mathcal{C}_{2}$ such that
$$\begin{aligned} \mathcal{C}_{1}\vec{\oplus}\mathcal{C}_{2^{*}}=&\operatorname{span}\left\{\left[\begin{array}{l}0 \\ 1  \\ \end{array}\right],\left[\begin{array}{l}1 \\  0\end{array}\right]\right\}\vec{\oplus}\operatorname{span}\left\{\left[\begin{array}{l}0 \\ 0\\ 1 \\ \end{array}\right]\right\}\\
=&\operatorname{span}\left\{\left[\begin{array}{l}0 \\ 1 \\0 \\ \end{array}\right],\left[\begin{array}{l}1 \\  0\\ 0\end{array}\right]\right\}+\operatorname{span}\left\{\left[\begin{array}{l}0 \\ 0\\ 1 \\ \end{array}\right]\right\}
=\mathbb{R}_{3}.
\end{aligned}$$
By Corollary \ref{C1}, the dimension-varying control system \eqref{E1} can realize the transient dynamics.

\section{Conclusion}\label{S6}
The realization problem of the transient dynamics of dimension-varying control systems has been investigated in this work.
By analyzing the controllable subspaces of linear systems on quotient space, a condition for the realization of the transient process for dimension-varying control systems has been given, based on which a new strategy for modelling the transient dynamics of dimension-varying control systems on quotient space is presented. As a result, theoretically we prove that the dynamic evolution between systems of different dimensions can be implemented by designing the control of  the transient dynamics model on quotient space. The correctness of theoretical results is verified by a numerical example.

\section*{References}
\bibliographystyle{elsarticle-num-names}
\bibliography{mybibfile}

\end{document}